\documentclass[dvi2pdf,epstopdf,12pt,a4paper,twoside]{amsart}

\usepackage{amsfonts}
\usepackage{amsmath}
\usepackage{amssymb}
\usepackage{amsthm}
\usepackage{caligr}
\usepackage{dsfont}
\usepackage{empheq}
\usepackage{enumerate}
\usepackage{enumitem}
\usepackage[multiple,symbol,stable]{footmisc}
\usepackage[hmargin={30mm,25mm},vmargin={30mm,30mm},includefoot]{geometry}
\usepackage{graphics} 
\usepackage{graphicx}
\usepackage[colorlinks]{hyperref}
\hypersetup{linkcolor=black, citecolor=black, urlcolor=black}
\usepackage{indentfirst}
\usepackage{latexsym}
\usepackage{parskip}
\usepackage{tikz}
\usetikzlibrary{positioning,fit,calc}
\usepackage[all]{xy}

\numberwithin{equation}{section}

\makeatletter
\renewcommand{\subsection}{\@startsection
{subsection}{2}{0mm}{\baselineskip}{-0.25cm}
{\normalfont\normalsize\em}}
\makeatother

\makeatletter
\def\gaps{\mathop{\operator@font Gaps}\nolimits}
\def\:={\mathrel{\mathop:}=}
\def\=:{=\mathrel{\mathop:}}
\makeatother

\def\neg1{\text{\boldmath$1$}}

\def\tt{\caligr T}

\newtheorem{theorem}{Theorem}[section]
\newtheorem{proposition}[theorem]{Proposition}

\newtheorem{lemma}[theorem]{Lemma}

{\theoremstyle{definition}

}
\newtheorem{algor}[theorem]{Algorithm}

\theoremstyle{remark}

\makeatletter
\def\moverlay{\mathpalette\mov@rlay}
\def\mov@rlay#1#2{\leavevmode\vtop{%
   \baselineskip\z@skip \lineskiplimit-\maxdimen
   \ialign{\hfil$\m@th#1##$\hfil\cr#2\crcr}}}
\newcommand{\charfusion}[3][\mathord]{
    #1{\ifx#1\mathop\vphantom{#2}\fi
        \mathpalette\mov@rlay{#2\cr#3}
      }
    \ifx#1\mathop\expandafter\displaylimits\fi}
\makeatother

\makeatletter
\renewcommand*\env@matrix[1][*\c@MaxMatrixCols c]{%
  \hskip -\arraycolsep
  \let\@ifnextchar\new@ifnextchar
  \array{#1}}
\makeatother

\begin{document}
    \author[F. Fornasiero]{F. Fornasiero}

\author[G. Tizziotti]{G. Tizziotti}

    \title{On Gr\"{o}bner basis for certain one-point AG codes}

\maketitle

\begin{abstract}
Heegard, Little and Saints worked out a Gr\"{o}bner basis algorithm for Hermitian codes and Farr\'{a}n, Munuera, Tizziotti and Torres extended such a result for codes on norm-trace curves. In this work we generalize such a result for codes arising from certain types of curves $\mathcal{X}$ over $\mathbb{F}_q$ with plane model $f(y)=g(x)$.
\end{abstract}

Keywords: AG codes; Gr\"{o}bner basis

MSC codes: 11T71; 13P10

\section{Introduction}
In the early 1980s, V.D. Goppa constructed error-correcting codes using algebraic curves, the called \emph{algebraic geometric codes} (AG codes), see \cite{goppa1} and \cite{goppa2}. The introduction of methods from algebraic geometry to construct good linear codes was one of the major developments in the theory of error-correcting codes. From that moment many studies and applications on this theory has emerged. In \cite{little2}, Little, Saints and Heegard introduced an encoding algorithm for a class of AG codes via Gr\"{o}bner basis, similar to the usual one for cyclic codes. This encoding method is efficient and also interesting from a theoretical point of view. It is known that the main drawback of Gr\"{o}bner basis is the high computational cost required for its calculation. Indeed, it is well known that the complexity of computing a Gr\"{o}bner basis is doubly exponential in general. But, in \cite{little}, using an appropriate automorphism of the Hermitian curve, Little et al. introduced the concept of \textit{root diagram} that allows to construct an algorithm for computing a Gr\"{o}bner basis with a lower complexity for one-point Hermitian codes. In \cite{FMTT}, the results of \cite{little} were extended to codes arising from the norm-trace curve, which is a generalization of the Hermitian curve. In both works, the one-point AG codes arising from curves over finite fields $\mathbb{F}_q$ with $q$ elements and the construction of the root diagram is made by using automorphisms whose order is equal to $q-1$. In this work, we will construct the root diagram, and consequently an algorithm for computing a Gr\"{o}bner basis, for codes arising from certain curves over $\mathbb{F}_q$ with automorphisms whose order divides $q-1$, thus we get results more general than those achieved previously. As examples, we have codes over the curves $y^q + y = x^{q^r + 1}$ and $y^q + y = x^m$, and codes over Kummer extensions, which have been applied in coding theory, see \cite{kondo} and \cite{matthews}, and \cite{quoos}, respectively.

This paper is organized as follows. In Section 2 we recall some background on Gr\"{o}bner basis for modules, AG codes and root diagram. In Section 3 we present a way to construct the root diagram for one-point AG codes $C$ arising from certain types of curves $\mathcal{X}$ over $\mathbb{F}_q$ with plane model $f(y)=g(x)$. In addition, we present the way to obtain the Gr\"{o}bner basis for such $C$. Finally, in Section 4 we present examples of those curves and the necessary informations to construct the root diagram and the Gr\"{o}bner basis studied in the previous section.

\section{Background}

\subsection{Gr\"{o}bner basis for $\mathbb{F}_{q}[t]$-modules}

We introduce some notations about Gr\"{o}bner basis for $\mathbb{F}_{q}[t]$-modules that we shall needed later.
For a complete treatment see \cite{adams} and \cite{oshea}. A \textit{monomial} $\mathbf{m}$ in the free $\mathbb{F}_{q}[t]$-module $\mathbb{F}_{q}[t]^r$ is an element of the form $\mathbf{m}=t^{i}\mathbf{e}_{j}$, where $i \geq 0$ and $\mathbf{e}_{1},\dots,\mathbf{e}_{n}$ is the standard basis of $\mathbb{F}_{q}[t]^r$. Fixed a monomial ordering, for all element $\mathbf{f} \in \mathbb{F}_{q}[t]^r$, with $\mathbf{f}\neq 0$, we may write $\mathbf{f}=a_1 \mathbf{m}_1 + \cdots + a_{\ell} \mathbf{m}_{\ell}$, where, for $1 \leq i \leq \ell$, $0 \neq a_i \in \mathbb{F}_q$ and $\mathbf{m}_i$ is a monomial in  $\mathbb{F}_{q}[t]^r$ satisfying $\mathbf{m}_1 > \mathbf{m}_2 > \ldots > \mathbf{m}_{\ell}$. The term $a_1\mathbf{m}_1$ is called \textit{leading term} of $\mathbf{f}$ and denoted by $LT(\mathbf{f})$, the coefficient $a_1$ and the monomial $\mathbf{m}_1$ are called \textit{leading coefficient}, $LC(\mathbf{f})$, and \textit{leading monomial}, $LM(\mathbf{f})$,
respectively. A \textit{Gr\"{o}bner basis} for a submodule $M \subseteq \mathbb{F}_{q}[t]^r$ is a set
$\mathcal{G} = \{ \mathbf{g}_{1}, \ldots , \mathbf{g}_{s} \}$ such that $\{ LT(\mathbf{g}_{1}), \ldots , LT(\mathbf{g}_{s}) \}$
generates the submodule $LT(M)$ formed by the leading terms of all elements in $M$.
The monomials in $LT(M)$ are called \textit{nonstandard} while those in the complement of
$LT(M)$ are the \textit{standard monomials} for $M$. We recall that every submodule $M \subseteq \mathbb{F}_{q}[t]^n$ has a Gr\"{o}bner
basis $\mathcal{G}$, which induces a a division algorithm: given $\mathbf{f} \in \mathbb{F}_{q}[t]^r$ there exist
$\mathbf{a}_{1},\ldots , \mathbf{a}_{s},\mathbf{R_{\mathcal{G}}} \in \mathbb{F}_{q}[t]^r$ such that
$\mathbf{f} = \mathbf{a}_{1} \mathbf{g}_{1} + \ldots + \mathbf{a}{s} \mathbf{g}_{s} + \mathbf{R_{\mathcal{G}}}$
(\cite{adams} Algorithm 1.5.1, or \cite{oshea} Theorem 3).

In this work we will use the POT (position over term) ordering over $\mathbb{F}_{q}[t]^r$ which is defined as follows.

Let $\{\mathbf{e}_1, \ldots , \mathbf{e}_r\}$ be the standard basis in $\mathbb{F}_{q}[t]^r$, with $\mathbf{e}_1> \ldots > \mathbf{e}_r$. The POT ordering on $\mathbb{F}_{q}[t]^r$ is defined by
$$
t^i \mathbf{e}_j > t^k \mathbf{e}_{\ell}
$$
if $j<\ell$, or $j=\ell$ and $i>k$.

We say that $\mathbf{f}\in \mathbb{F}_{q}[t]^r$ is \textit{reduced} with respect to a set $P = \{ \mathbf{p}_{1},\dots, \mathbf{p}_{l}\}$
of non-zero elements in $\mathbb{F}_{q}[t]^r$ if $\mathbf{f} = \mathbf{0}$ or no monomial in $\mathbf{f}$ is divisible by a
$LM(\mathbf{p}_{i})$, $i=1, \ldots , l$. A Gr\"{o}bner basis $\mathcal{G} = \{ \mathbf{g}_{1}, \ldots , \mathbf{g}_{s} \}$ is
\textit{reduced} if $\mathbf{g}_{i}$ is reduced with respect to $\mathcal{G} - \{\mathbf{g}_{i}\}$ and $LC(\mathbf{g}_{i})=1$ for all
$i$, and \textit{non-reduced} otherwise. Every submodule of $\mathbb{F}_{q}[t]^r$ has a unique reduced Gr\"{o}bner basis (see \cite{adams}, Theorem 3.5.22).

\medskip

\subsection{Linking AG codes and $\mathbb{F}_{q}[t]$-modules} \label{AG codes and modules} Let $\mathcal{X}$ be a projective, non-singular, geometrically irreducible, algebraic curve of genus $g>0$ defined over $\mathbb{F}_{q}$. Let $P_1,\dots,P_n,Q_1,\dots,Q_{\ell}$ be $n+ \ell$ distinct rational points on
$\mathcal{X}$ and $m_{1}, \ldots, m_{\ell}$ be integers. Consider the divisors $D=P_{1}+\cdots+P_{n}$, $G=m_1Q_1+\dots+m_{\ell}Q_{\ell}$. The {\em algebraic geometry code} (AG
code) $C_{\mathcal{X}}(D,G)$ associated to the curve $\mathcal{X},$ is defined as
\begin{equation}\label{defi AG code}
C_{\mathcal{X}}(D,G):=\{(f(P_1),\ldots,f(P_n))\in \mathbb{F}_q^n : f \in \mathcal{L}(G)\}\, ,
\end{equation}
where $\mathcal{L}(G)$ is the space of rational functions $f$ on $\mathcal{X}$ such that $f=0$ or $\mbox{div}(f)+G\ge 0$, where $\mbox{div}(f)$ denote the (principal) divisor of the function $f \in \mathcal{L}(G)$. The number $n=|Supp(D)|$ is the length of $C_{\mathcal{X}}(D,G)$, where $Supp(D)$ denotes the support of the divisor $D$, and the dimension of $C_{\mathcal{X}}(D,G)$ is its dimension as an $\mathbb{F}_q$-vector space, which is generally denoted by $k$. The elements in $C_{\mathcal{X}}(D,G)$ are called \textit{codewords}. If $G=aP$, for some rational point $P$ on $\mathcal{X}$, and $D$ is the sum of the all others rational points on $\mathcal{X}$ the AG code $C_{\mathcal{X}}(D,\lambda P)$ is called \textit{one-point AG code}. For more details about AG codes, see e.g. \cite{vanlint}.

Let $S_{n}$ be the symmetric group. $S_{n}$ acts on ${\mathbb{F}}_{q}^n$ via $\tau(a_{1}, \ldots , a_{n})= (a_{\tau(1)},
\ldots , a_{\tau(n)})$, where $\tau \in S_{n}$. The {\em automorphism group} of a code $C$ is defined as
$$\mbox{Aut}(C):=\{\sigma \in S_{n}:\sigma(C)=C \}\;.$$

In \cite{goppa2}, Goppa already observed that the underlying algebraic curve induces automorphism of the associated AG codes as follows.

\medskip

\begin{proposition} \label{prop automorphism}
Let $\mbox{Aut}({\mathcal{X}})$ be the automorphism group of $\mathcal{X}$ over $\mathbb{F}_{q}$ and consider the subgroup
   $$
Aut_{D,G}({\mathcal{X}})= \{\sigma \in\mbox{Aut}({\mathcal{X}})
\mbox{ : } \sigma(D)=D \mbox{ and } \sigma(G)=G \}\;.
   $$
Each $\sigma \in \mbox{\rm Aut}_{D,G}({\mathcal{X}})$ induces an automorphism
of $C_{\mathcal{X}}(D,G)$ by $$\widehat{\sigma} (f(P_{1}), \ldots , f(P_{n}))= (f(\sigma(P_{1})), \ldots,f(\sigma(P_{n})))\;.$$
\end{proposition}

\medskip

Assume that $\mathcal{X}$ has a nontrivial automorphism $\sigma \in Aut_{D,G} (\mathcal{X})$ and let $H$ be the cyclic subgroup of $Aut (\mathcal{X})$ generated by $\sigma$. Let $Supp(D)=O_1 \cup \ldots \cup O_r$ be the decomposition of the support of $D$ into disjoint orbits under the action of $\sigma$. Then, by Proposition \ref{prop automorphism}, the entries of codewords in $C_{\mathcal{X}}(D,G)$ corresponding to the points in each $O_i$ are permuted cyclically by $\sigma$. We will denote $\sigma^0 = Id$, where $Id$ is the identity automorphism, and, for a positive integer $j$, $\sigma^j = \underbrace{\sigma \circ \sigma \circ \ldots \circ \sigma}_{j}$. In this way, for each $i=1,\ldots, r$, by choosing any one point $P_{i,0} \in O_i$, we can enumerate the other points on $O_i$ as $P_{i,j} = \sigma^{j}(P_{i,0})$, where $j$ runs from $0$ to $|O_i|-1$.  Using this fact, we get the following result.

\begin{lemma} \label{lema submodulo}
Let $C_{\mathcal{X}}(D,G)$ be an AG code arising from $\mathcal{X}$ over $\mathbb{F}_{q}$. Suppose that $\mathcal{X}$ has a nontrivial automorphism $\sigma \in Aut_{D,G} (\mathcal{X})$. If $Supp(D)=O_1 \cup \ldots \cup O_r$ is the decomposition of the support of $D$ into disjoint orbits under the action of $\sigma$, then there is an one-to-one correspondence between $C_{\mathcal{X}}(D,G)$ and a submodule $\overline{C}$ of the free module $\mathbb{F}_{q}[t]^r$.
\end{lemma}

\begin{proof}
Suppose that $Supp(D)=O_1 \cup \ldots \cup O_r$ is the decomposition of the support of $D$ into disjoint orbits under the action of $\sigma$. For each $i=1,\ldots,r$, let $O_i = \{P_{i,0}, \ldots, P_{i,|O_i|-1}\}$, where for each $P_{i,j} \in O_i$ we have that $P_{i,j}=\sigma^j(P_{i,0})$ be as above, and let $h_i(t) = \sum_{j=0}^{|O_i|-1} f(P_{i,j})t^j$. 

The $r$-tuples $(h_1(t), \ldots, h_r(t))$ can be seen also as an element of the $\mathbb{F}_{q}[t]$-module $A = \bigoplus_{i=1}^{r}\mathbb{F}_{q}[t]/ \langle t^{|O_{i}|} -1 \rangle$. So, the collection $\tilde{C}$ of $r$-tuples obtained from all $f\in \mathcal{L}(G)$ is closed under sum and multiplication by $t$. Define $\overline{C} := \pi^{-1}(\tilde{C})$, where $\pi$ is the natural projection from $\mathbb{F}_{q}[t]^r$ onto $\bigoplus_{i=1}^{r} \mathbb{F}_{q}[t]/ \langle t^{|O_{i}|} -1 \rangle$. Thus, we get an one-to-one correspondence between $C_{\mathcal{X}}(D,G)$ and $\overline{C} \leq \mathbb{F}_{q}[t]^r$. $\square$
\end{proof}

By the previous lemma, an AG code $C_{\mathcal{X}}(D,G)$ can be identified to a submodule $\overline{C}\leq \mathbb{F}_{q}[t]^r$ and the standard theory of Gr\"{o}bner basis for modules may be applied.

Suppose that $C_{\mathcal{X}}(D,G)$ has length $n$ and dimension $k$. A Gr\"{o}bner basis $\mathcal{G} = \{ \mathbf{g}^{(1)}, \ldots , \mathbf{g}^{(r)} \}$ for $\overline{C}\leq \mathbb{F}_{q}[t]^r$ with exactly $r$ elements allows us to obtain a systematic encoding of $C$. 
Since $\{ LT(\mathbf{g}^{(1)}), \ldots,LT(\mathbf{g}^{(r)}) \}$ generates $LT(\overline{C})$,
then the nonstandard monomials appearing in the $r$-uples $(h_{1}(t), \ldots , h_{r}(t))$ can be obtained from the $\mathbf{g}^{(i)}$'s.
By ordering these monomials in decreasing order we obtain the so-called
\textit{information positions} of $(h_{1}(t), \ldots , h_{r}(t))$, which are the first $k$ monomials
$\mathbf{m}_{l} = t^{i_{l}}\mathbf{e}_{j_{l}}$, $l=1, \ldots , k$.
Let $VC(h_{1}(t), \ldots , h_{r}(t))$ be the vector of coefficients of the terms of
$(h_{1}(t), \ldots , h_{r}(t))$ listed in the POT order. We have the following systematic encoding algorithm:

\begin{algor}\label{encodingalgor}
$\,$
$\mbox{}$\newline
{\bf Input:} A Gr\"{o}bner basis $\mathcal{G}$, monomials
$\{ \mathbf{m}_{1}, \ldots , \mathbf{m}_{k} \}$ and $\mathbf{w}=(w_{1}, \ldots , w_{k}) \in \mathbb{F}_{q}^k$. \newline
{\bf Output:} $c(\mathbf{w}) \in C=C(\mathcal{X},D,G)$.
$\mbox{}$\newline
$\mbox{ }$ 1. Set $f:= w_{1} \mathbf{m}_{1}+\dots+w_{k} \mathbf{m}_{k}$. \newline
$\mbox{ }$ 2. Compute $f= \mathbf{a}_{1} \mathbf{g}^{(1)} + \ldots + \mathbf{a}_{r} \mathbf{g}^{(r)} + \mathbf{R_{\mathcal{G}}}$. \newline
$\mbox{ }$ 3. Return $c(\mathbf{w}):=VC(f-\mathbf{R_{\mathcal{G}}})$.

\end{algor}

This method is more compact compared with the usual encoding via generator matrix. The total amount of computation is roughly the same
and the amount of necessary stored data is lower in this method, of order $r(n-k)$ against $k(n-k)$ when encoding via generator matrix. More details about this encoding algorithm can be found in
\cite{little2}.

\subsection{The root diagram}
Let $\mathcal{X}$ be as in the previous subsection. Suppose that the one-point AG code $C=C_{\mathcal{X}}(D,\lambda P)$ has an automorphism $\sigma$ that fixing the divisors $D$ and $G=\lambda P$. Suppose also that the order of $\sigma$ is equal to $s$, with $s = d( q-1)$ for some $d \in \mathbb{N}$. Let $\overline{C}$ be the submodule of $\mathbb{F}_{q}[t]^r$ associated to $C$ by the automorphism $\sigma$. Using the POT ordering we can get that a Gr\"{o}bner basis $\mathcal{G}=\{\mathbf{g}_1, \ldots , \mathbf{g}_r\}$ for $\overline{C}$ such that $\mathbf{g}_i = (0,\ldots , 0 , g_i^{(i)}(t), g_i^{(i+1)}(t), \ldots , g_i^{(r)}(t))$, for all $i=1,\ldots,r$, see [\cite{little2}, Proposition II.B.4].

Note that, if $deg(g_i^{(i)}(t))=d_i$, then $g_i^{(i)}(t)$ has $d_i$ distinct roots in $\mathbb{F}_q^{\ast}=\mathbb{F}_q \setminus \{0\}$. In fact, let ${\mathbf{q}}_{i}=(t^{|O_{i}|}-1){\mathbf{e}}_{i}$. Note that ${\mathbf{q}}_{i}\in\pi^{-1}(0,\ldots ,0)$ and we have that ${\mathbf{q}}_{i}\in\overline{C}$. Since $|O_i|$ divides $s$ and $s$ divides $q-1$, follows that $t^{|O_{i}|}-1$ divides $t^{q-1}-1 = \prod_{a \in \mathbb{F}_q^{\ast}}(t-a)$. Now, $LT({\mathbf{g}}^{(i)})=g_{i}^{(i)}(t)$ divides $LT({\mathbf{q}}_{i})=t^{|O_{i}|}-1$, and the claim follows from the fact $t^{q-1}-1$ has $q-1$ distinct roots in $\mathbb{F}_q$.

For $i=1, \ldots , r$, let ${\mathcal{R}}_{i}\subseteq {\mathbb{F}}_{q}^*$ be the set of roots of $t^{|O_{i}|}-1$.
By a \textit{root diagram} ${\mathcal{D}_{C}}$ for the code $C$, we mean a table with $r$ rows. For each $i$, the boxes on the $i$-th row correspond to the elements of ${\mathcal{R}}_{i}$.
We mark the roots of $g_{i}^{(i)}(t)$ on the $i$-th row with a $X$ in the corresponding box.

By Proposition II.C.1 in \cite{little2}, there is a $\mathbb{F}_q$-basis for $C$ in one-to-one correspondence with the nonstandard monomials in $\overline{C}$. That is, terms of the form $t^{\ell}\mathbf{e}_j$ appearing as leading terms of some element of $\overline{C}$, with $\ell \leq |O_j|-1$. Now, if there are $m_j$ empty boxes on row $j$ of the root diagram, then $g_{j}^{(i)}(t)$ has $|O_j| - m_j$ roots and $LT({\mathbf{g}}^{(j)})=t^{|O_j|-m_j}$. So, we obtain $m_j$ nonstandard monomials $t^{\ell}\mathbf{e}_j$. This fact gives us the following important result.

\begin{proposition}{\rm (\cite{little}, Proposition 2.3)} \label{propdimensao}
The dimension of the code $C$ is equal to the number of empty boxes in the root diagram $\mathcal{D}_{C}$.
\end{proposition}

\section{Gr\"{o}bner basis for certain AG codes}

Finding a Gr\"obner basis is hard in general. Next, we will see that for certain codes AG this task is simplified by using the concept of root diagram.

Let $\mathcal{X}$ be as in the previous section and let $\mathbb{F}_{q}(\mathcal{X})$ be the field of rational functions on $\mathcal{X}$. For a rational point $P$ on $\mathcal{X}$ let
$$
H(P):= \{n \in \mathbb{N}_{0} \mbox{ ; } \exists f \in \mathbb{F}_{q}(\mathcal{X}) \mbox{ with } div_{\infty}(f) = n P \},
$$
where $\mathbb{N}_{0}$ is the set of nonnegative integers and $div_{\infty}(f)$ denotes the divisor of poles of $f$. The set $H(P)$ is a numerical semigroup, called \textit{Weierstrass semigroup} of $\mathcal{X}$ at $P$ and its complement $G(P) = \mathbb{N}_{0} \setminus H(P)$ is called \textit{Weierstrass gap set} of $P$. As an important result, the cardinality of $G(P)$ is equal to genus $g$ of $\mathcal{X}$, see Theorem 1.6.8 in \cite{stichtenoth2}.

Let $\mathcal{X}_{a,b}$ be the curve defined over $\mathbb{F}_q$ by affine equation $f(y)=g(x)$, where $f(T),g(T) \in \mathbb{F}_q[T]$, $deg(f)=a$ and $deg(g)=b$, with $a<b$ and $gcd(a,b)=1$. Furthermore, suppose that  $div_{\infty}(x)=aP$ and $div_{\infty}(y)=bP$, for some point on $\mathcal{X}_{a,b}$, and that there exists $\sigma \in Aut_{D,G}(\mathcal{X}_{a,b})$, where $G=\lambda P$ for some positive integer $\lambda$, given by $\sigma(x)=\alpha x$ and $\sigma(y)=\alpha^t y$, for some positive integer $t$ and some $\alpha \in \mathbb{F}_{q}^{\ast}$ with order equal to $ord(\alpha):=\nu$. Finally, assume that $H(P)=\langle a,b \rangle$.

Consider the one-point AG code $C_{\mathcal{X}_{a,b}}(D,\lambda P)$. Let $D=P_1+ \ldots + P_n$ and $Supp(D)=O_1 \cup \ldots \cup O_r \cup O_{r+1} \cup O_{r+s}$ be the decomposition of the support of $D$ into disjoint orbits under the action of $\sigma$. Note that, by definition of $\sigma$, if $Q=(0,\eta)\in O_i$, for some $\eta \in \mathbb{F}_q$, then $O_i = \{ (0,\eta), (0, \alpha^t \eta), \ldots , (0,\alpha^{t.t_i}) \}$, where $t_i$ is the smallest nonnegative integer such that $\alpha^{t.(t_i+1)}=1$. Analogously, if $Q=(\omega,0)\in O_i$, for some $\omega \in \mathbb{F}_q$, then $O_i = \{ (\omega,0), (\alpha \omega , 0), \ldots , (\alpha^{\nu-1} \omega,0) \}$. Let $O_{r+1}, \ldots , O_{r+s}$ be the orbits that contains rational points of the form $(0,\eta)$ or $(\omega,0)$. We will work with the first $r$ rows of the root diagram $\mathcal{D}_{C}$ for the code $C_{\mathcal{X}_{a,b}}(D,\lambda P)$, the results for the last $s$ rows are similar can be obtained in particular cases. For each $i=1,\ldots,r$, suppose that $O_i = \{ P_{i,0}, P_{i,1}, \ldots, P_{i,|O_i|-1}  \}$, where $P_{i,0}=(x_i,y_i)$, with $x_i \neq 0$ and $y_i\neq0$, and $P_{i,j}=\sigma^j(P_{i,0})=(\alpha^j x_i , \alpha^{jt}y_i)$. So, by the definition of $\sigma$ follows that $|O_1|= \ldots = |O_r|=ord(\alpha)=\nu$. Assume that, for each $i=1,\ldots, r$, there exists polynomials $M_i(y)$ such that the orbit $O_{i}$ is the intersection of $\mathcal{X}$ with the curve $M_{i}(y)=0$ and, for all $i$, $M_{i}(y)$ is a non-zero constant when restricted to each of the orbits $O_{k}$, $k\neq i$. For $1 \leq i \leq r$ and $0 \leq j \leq |O_i|-1=\nu -1$, assume also that there are polynomials $B_{i,j}(x,y)$ such that $B_{i,j}(x,y)$ vanishes at each point of $O_{i}$ except $P_{i,j}$.

\begin{lemma} \label{lemma 1}
For  $i=1,\ldots,r$ and $j=0, \ldots, |O_i|-1$, let $M_{i}(y)$ and $B_{i,j}(x,y)$ be as above. Then, $div_{\infty}(M_i)=(\rho_1 b) P$ and $div_{\infty}(B_{i,j})=(\rho_2 a+\rho_3 b)P$, where $\rho_1, \rho_2$ and $\rho_3$ are non-negative integers.
\end{lemma}

\begin{proof}
We have that $div_{\infty}(x)=aP$ and $div_{\infty}(y)=bP$. Then, the result follows from the fact that $M_i(y)$ and $B_{i,j}(x,y)$ are polynomials. $\square$
\end{proof}

Let $\rho_1, \rho_2$ and $\rho_3$ be as the previous lemma. So, for $\displaystyle \lambda \leq (\rho_2 a+\rho_3 b) + r(\rho_1 b)$, we can get the following information about the root diagram $\mathcal{D}_{C}$.

\begin{proposition}\label{123eta2}
Let $C_{\mathcal{X}_{a,b}}(D,\lambda P)$ and $\sigma$ be as above. Let $\mathcal{D}_{C}$ be the root diagram for $C_{\mathcal{X}_{a,b}}(D,\lambda P)$. Fix $i$, $1\leq i \leq r$, and let $\rho_1$, $\rho_2$ and $\rho_3$ be as above. Therefore,

1) if $\displaystyle \lambda \geq (i-1)(\rho_1 b)$, then the $i$-th row of $\mathcal{D}_{C}$ is not full, in the sense that not every boxes composing the $i$-th row are marked with $X$;

2) if $\displaystyle \lambda \geq (\rho_2 a+\rho_3 b) + (i-1)(\rho_1 b)$, then the row is empty, in the sense that none of the boxes composing the $i$-th row is marked with $X$.
\end{proposition}

\begin{proof}

Let $\overline{C}\leq \mathbb{F}_{q}[t]^r$ be the submodule associated to $C_{\mathcal{X}_{a,b}}(D,\lambda P)$.

1) Suppose that $\displaystyle \lambda \geq (i-1)(\rho_1 b)$. By Lemma \ref{lemma 1}, the function
$$
F_{i}(x,y)= M_1(x,y)\cdots M_{i-1}(x,y)
$$
belongs to $L(\lambda P)$ and hence $ev(F_i)\in C_{\mathcal{X}_{a,b}}(D,\lambda P)$.
By computing $ev(F_{i})$, we observe that $\overline{C}$ contains an element of the form
$(0,\ldots , 0, h_{i}(t), \ldots , h_{r}(t))$ with $i-1$ zeroes and $h_{i}(t)=\sum_{j=0}^{|O_{i}|-1}F_i(P_{i,j})t^j.$
Since
\[
F_{i}(P_{i,j})=M_1(P_{i,j})\cdots M_{i-1}(P_{i,j})= \mbox{ constant } c \neq 0\, ,
\]
we have $h_{i}(t)=c.\sum_{j=0}^{|O_{i}|-1} t^j$
and thus $h(1) \neq 0$ as $|O_{i}|$ divides $q-1$.
Therefore the $i$-th row of $\mathcal{D}_{C}$ is not full, since $g_{i}^{(i)}(t)$ divides $h_{i}(t)$.

\medskip

2) Now, suppose $\displaystyle \lambda \geq (\rho_2+\rho_3 b) + (i-1)(\rho_1 b)$. So, by Lemma \ref{lemma 1}, $G_i(x,y)=B_{i,0}(x,y)F_i(x,y) \in L(\lambda P)$ and $G_{i}(Q)=0$ for $Q\in O_{1}\cup O_{2} \cup \ldots \cup O_{i-1}$.
Moreover, $G_{i}(Q)=0$ for all $Q \in O_{i} \setminus \{P_{i,0}\}$.
Then the element of $\overline{C}$ corresponding to $ev(F'_i)$ verifies
$h_{1}(t)=h_{2}(t)= \ldots = h_{i-1}(t)=0$ and $h_{i}(t)=G_{i}(P_{i,0})=c \neq 0$.
Thus, $\overline{C}$ contains the element $(0,\ldots ,0, c, h_{i+1}(t), \ldots , h_{r}(t))$
and follows that the $i$-th row of $\mathcal{D}_{C}$ is empty. $\square$
\end{proof}

Let $N$ be the number of rational points on $\mathcal{X}_{a,b}$, by Riemman-Roch Theorem,  follows that if $\lambda < N$, then the dimension of the one-point AG code $C_{\mathcal{X}_{a,b}}(D,\lambda P)$ is equal to the dimension of the Riemann-Roch space $L(\lambda P)$. In this case, we complete the informations about the root diagram $\mathcal{D}_C$.

\begin{theorem} \label{main theorem} Let $\mathcal{D}_C$ be the root diagram for $C_{\mathcal{X}_{a,b}}(D,\lambda P)$. If there is $i \in \{1, \ldots ,r \}$ such that
$$
(i-1)(\rho_1 b) \leq \lambda <  (\rho_2 a+\rho_3 b) + (i-1)(\rho_1 b),
$$
	
then the $i$-th row of $\mathcal{D}_C$ is neither full, nor empty, and the complement of the set of roots marked on row $i$ of the diagram is the set

$E_i = \{\alpha^{-(\beta+\gamma b)}\in\mathbb{F}_{q}^*\mid 0\leq\beta\leq b -1 \mbox{, } 0\leq \gamma\leq \rho_1 -1 \mbox{, }  (i-1)(\rho_1 b) +\beta a+ \gamma b\leq \lambda\}$.

\end{theorem}
\begin{proof}

Let $\overline{C}\leq \mathbb{F}_{q}[t]^r$ be the submodule associated to $C_{\mathcal{X}_{a,b}}(D,\lambda P)$. Let $D_i \subset \mathbb{F}_{q}^{\ast}$ be the set of non marked boxes in row $i$, where $1 \leq i \leq r$. We will show that $D_{i} = E_{i}$. Let $F_i(y)$ be as in the previous proposition and consider $f_i(x,y) = F_i(y)x^{\beta}y^{\gamma}$. By Lemma \ref{lemma 1} and the conditions over $\beta$ and $\gamma$ given in the definition of $E_i$, we have that $f_i(x,y) \in L(\lambda P)$. So, associated to $f_i(x,y)$ we get an element $h=(h_1(t), \ldots , h_{r}(t)) \in \overline{C}$. Since $F_i(Q)=0$ for all $Q \in O_1 \cup \ldots \cup O_{i-1}$, follows that $h_k(t)=0$, for $k=1,\ldots,i-1$. Let $P_{i,j}=\sigma(P_{i,0})=(\alpha^j x_i,\alpha^{tj}y_i) \in O_i$. Thus, $f_i(P_{i,j})= F_i(P_{i,j})\alpha^{j\beta} x_{i}^{\beta}\alpha^{tj\gamma}y_{i}^{\gamma} = F_i(P_{i,j})x_{i}^{\beta} y_{i}^{\gamma} \alpha^{j(\beta + t \gamma)}$. Now, $F_i(P_{i,j})$, $x_{i}^{\beta}$ and $y_{i}^{\gamma}$ are all non-zero constants and independents of $j$. Taking $b_i=F_i(P_{i,j})x_{i}^{\beta} y_{i}^{\gamma} \neq 0$, we have

$\displaystyle h_i(t) = \sum_{j=0}^{|O_i|-1}f_i(P_{i,j})t^j = [|O_i|-1]b_i \sum_{j=0}^{|O_i|-1} (\alpha^{\beta + t \gamma} t)^j$ whose roots are all distinct from $\alpha^{-(\beta + t \gamma)}$. Consequently, $\alpha^{-(\beta + t \gamma)}$ is not a root of $g_{i}^{(i)}(t)$ and hence $E_{i} \subseteq D_{i}$.

By Proposition \ref{propdimensao}, $dim (C_{\mathcal{X}_{a,b}}(D,\lambda P)) = \sum \sharp D_i$. Since $H(P) = \langle a,b \rangle$ and $\lambda < N$, we have that $dim(C_{\mathcal{X}_{a,b}}(D,\lambda P)) = \sharp \{ (\beta , \gamma) \in \mathbb{N}^2 \mbox{ ; } 0 \leq \beta \leq  b-1 \mbox{ and } \beta a + \gamma b \leq \lambda \}$.

Let $\widehat{E}_i=\{(\beta , \gamma) \in \mathbb{N}^2 \mid 0\leq\beta\leq b-1,  0\leq \gamma\leq \rho_1 -1, \,\, (i-1)(\rho_1 b)+\beta a+ \gamma b\leq \lambda\}$. Thus, $\sharp \{ (\beta , \gamma) \in \mathbb{N}^2 \mbox{ ; } 0 \leq \beta \leq  b-1 \mbox{ and } \beta a + \gamma b \leq \lambda \} = \sum \sharp \widehat{E}_i$ and, since $\sum \sharp  \widehat{E}_i = \sharp \sum E_i$, follows that $\sum \sharp D_i = \sum \sharp E_i$. Therefore, $E_i = D_i$. $\square$
\end{proof}

Let $F_i(y)$ be as above. Then, we have that $F_i(Q)=c_i \in \mathbb{F}_{q}^{\ast}$, for all $Q \in O_i$. With the conditions of the above theorem, fix an index $i$, $1 \leq i \leq r$, where the row $i$ of $\mathcal{D}_C$ is neither full, nor empty. Let $\alpha^{k_{1}}, \alpha^{k_{2}}, \ldots ,\alpha^{k_{\ell}}$ be the roots marked on the row $i$ and let $p(t)=\prod_{j=1}^{\ell}(t-\alpha^{k_{j}})$ be the unique monic polynomial of degree $\ell$ with these roots. Note that, including zeroes for powers of $t$ higher than the number of roots, we can write $p(t)= \sum_{j=0}^{|O_i|-1} a_j t^j$, where $a_j=0$ for $j> \ell$. Consider the function
$$
\displaystyle f_{i}(x,y)=  \frac{F_i(y)}{c_i} \left(
\sum_{j=0}^{|O_{i}|-1} a_{j}\frac{B_{i,j}(x,y)}{B_{i,j}(P_{i,j})}
\right)
$$

Then, by definition of $F_i(y)$ and $B_{i,j}(x,y)$, it is clear that $f_{i}(x,y)\in L(\lambda P)$ and its associated module element $\mathbf{h} \in \overline{C}$ has $i-1$ leading zero components and $i$-th component $h_{i}(t)$ equal to $p(t)$.

So, using the same procedures used in \cite{little} and \cite{FMTT}: 

$\circ$ \texttt{RootDiagram[$i$]}:  returns a list of the roots
corresponding to the marked boxes in line $i$ of $\mathcal{D}_C$;

$\circ$ \texttt{Boxes[$i$]}: the number of boxes in row $i$ of $\mathcal{D}_C$, that is $\mbox{\tt Boxes[$i$]}=|O_i|$;

$\circ$ \texttt{Evaluate[$i, point$]}: a procedure which takes as input the coefficients $\{ a_{k}\}$ of the unique monic polynomial
over ${\mathbb{F}}_{q}$ having the marked elements on a row number $i$ as roots and a point $P_{i,j}$ on $O_{i}$, and
evaluates the function $f_{i}(x,y)$ as above at a $point$ $P_{i,j}$;

we get an analogous algorithm that computes a non-reduced POT Gr\"{o}bner basis for the submodule $\overline{C}$ associated to $C_{\mathcal{X}_{a,b}}(D,\lambda P)$ and thus to apply the systematic encoding given in Subsection \ref{AG codes and modules} to the AG codes $C=C_{\mathcal{X}_{a,b}}(D,\lambda P)$.

\begin{algor}\label{etalgoritmo}

$\,$

{\rm
\textbf{Input:} the root diagram $\mathcal{D}_C$, the $N$ rational points $P_{i,j}$ of $Supp(D)=O_1 \cup \ldots \cup O_r \cup O_{r+1} \cup O_{r+s}$. \newline
\textbf{Output:} a non-reduced Gr\"obner basis $\mathcal{G}=\{ \mathbf{g}^{(1)},\mathbf{g}^{(2)},\ldots,\mathbf{g}^{(r+s)} \}$ of $\overline{C}$.

\begin{tabbing}
1. $\mathcal{G}:=\{\}$ \\
2. {\bf for} $i$ from 1 to $r+s$ {\bf do} \\
3. \hspace*{3mm} \= {\bf if} {\tt |RootDiagram[$i$]|}$<${\tt Boxes[$i$]} {\bf then} \\
4. \> \hspace*{3mm} \= {\bf for} $k$ from $1$ to $r+s$ {\bf do} \\
5. \> \> \hspace*{3mm} \= $g_{k}^{(i)}:=0$ \\
6. \> \> \> {\bf if} $k\geq i$ {\bf then} \\
7. \> \> \> \hspace*{3mm} \= {\bf for} $j$ from $0$ to $\mbox{\tt Boxes[$k$]}-1$ {\bf do} \\
8. \> \> \> \> \hspace*{3mm} \= $g_{k}^{(i)}:=g_{k}^{(i)}+\mbox{\tt Evaluate[$i$,$P_{k,j}$]}\,t^j\,\mathbf{e}_{k}$ \\
9. \> \> \> \> {\bf end for} \\
10. \> \> \> {\bf end if} \\
11. \> \> {\bf end for} \\
12. \> {\bf else} \\
13. \> \> $\mathbf{g}^{(i)}:=(t^{ \mbox{\small \tt Boxes[$i$]}}-1)\,\mathbf{e}_{i}$ \\
14. \> {\bf end if} \\
15. \> $\mathcal{G}:=\mathcal{G}\cup \{\mathbf{g}^{(i)}\}$ \\
16. {\bf end for} \\
17. {\bf return} $\mathcal{G}$
\end{tabbing}
}

\end{algor}

We note that this algorithm has the same computational complexity as the original one developed by Little, Saints and Heegard
in \cite{little}. It is much lower than the complexity of general Gr\"{o}bner basis algorithms, since we only make use of
interpolation problems and evaluation of functions. In particular we do not use divisions nor reductions that would increase the
complexity, as in the case of Buchberger's algorithm.


\section{Examples}

\subsection{The curve $\mathcal{X}_{q^{2r}}$} Let $\mathcal{X}_{q^{2r}}$ be the curve defined over $\mathbb{F}_{q^{2r}}$ by the affine equation
   $$
y^q + y = x^{q^r + 1} ,
  $$
where $q$ is a prime power and $r$ an odd integer. Note that when $r=1$ the curve is just the Hermitian curve.  The curve $\mathcal{X}_{q^{2r}}$ has genus $g=q^{r}(q-1)/2$, one single singular point $P_{\infty}=(0:1:0)$ at infinity and others $q^{2r+1}$ rational points. Thus, this curve is a maximal curve over $\mathbb{F}_{q^{2r}}$ because its number of rational points equals the upper Hasse-Weil bound, namely equals $q^ {2r} + 1 + 2gq^r$. Furthermore, $H(P_{\infty}) = \langle q, q^{r}+1 \rangle$, see \cite{st}, and
\begin{equation} \label{automorfismo sigma}
\begin{array}{cccc}
	\sigma: & x & \mapsto & \alpha x\\
	& y& \mapsto & \alpha^{q^r +1}y
\end{array}
\end{equation}

with $\alpha \in \mathbb{F}_{q^{2r}}^{\ast}$ such that $\alpha^{(q^r+1)(q-1)}=1$, is an automorphism of $\mathcal{X}_{q^{2r}}$, see \cite{kondo}. Note that $\sigma$ has order $(q^r+1)(q-1)$. So, the order of $\sigma$ divides $q^{2r}-1$.

Note that under the action of the automorphism $\sigma$ above the $q^{2r+1}$ rational points on $\mathcal{X}_{q^{2r}}$ are disposed in $q(q^{r-1}+\cdots+q)+2$ orbits, where $q(q^{r-1}+\cdots+q)$ of them has length $(q^r+1)(q-1)$ and the remaining two orbits, one has length $q-1$ and the other has length $1$. In fact, for the definition of the automorphism $\sigma$, it is clear that:

$\cdot$ $\sigma(0,0)=(0,0)$, and so we have a one orbit with a single point;

$\cdot$ all the $q-1$ rational points $(0,b)$, with $b \neq 0$, form an orbit with length $q-1$, since $\sigma(0,b)=(0,\alpha^{q^r +1}b)$ and $\alpha \in \mathbb{F}_{q^{2r}}^{\ast}$ is such that $\alpha^{(q^r+1)(q-1)}=1$;

$\cdot$ the others $q^{2r+1}-q = q(q^r+1)(q^r-1)$ rational points $(x,y)\in \mathcal{X}_{q^{2r}}$, with $x\neq 0$ and $y \neq 0$, are arranged in $q(q^{r-1}+\cdots+q)$ orbits of length $(q^r+1)(q-1)$.

\medskip

Let $r=q(q^{r-1}+\cdots+q)$ and $\alpha$ be as in (\ref{automorfismo sigma}). Let $\mathbb{F}_{q^{2r}}^{\ast} = \langle a \rangle$ and $t \in \{ 0,1, \ldots , q^{2r}-2 \}$ be such that $\alpha = a^t$. So, given $P_{i,0}=(a^{t_i},a^{l_i}) \in O_i$, the others points $P_{i,j}$ on $O_i$ are $P_{i,j} =\sigma^j(P_{i,0})=(a^{t_i+jk},a^{l_i+jk(q^r+1)})$, with $j \in \{ 1,\ldots, (q^r+1)(q-1)-1 \}$. Then, for $i=1,\ldots,r$ and $j=0,\ldots,(q^r+1)(q-1)-1$, we get
\begin{equation} \label{eq1}
M_i(y):=\prod_{j=0}^{q-2}(y-a^{l_i+jk(q^r+1)})=y^{q-1}-a^{l_i(q-1)},
\end{equation}

and

\begin{equation}
B_{i,j}(x,y);= \prod_{s=1}^{q-2}(y-a^{l_i+k(q^r+1)(j+s)})\prod_{s=1}^{(q^r+1)-1}(x-a^{t_i+k(j+s)}).
\end{equation}

Since $div_{\infty}(x) = q P_{\infty}$ and $div_{\infty}(y) = (q^r + 1)P_{\infty}$, we have that

\medskip

$\bullet$ $div_{\infty}(M_i(y)) = (q-1)(q^r+1) P_{\infty}$, that is, $M_{i}(y) \in L((q-1)(q^r+1) P_{\infty})$, for all $i=1,\ldots,r$;

\medskip

$\bullet$ $div_{\infty}(B_{i,j}(x,y)) = ((q-2)(q^r+1)+q((q^r+1)-1))P_{\infty}$, that is, $B_{i,j} \in L(q.q^r+(q-2)(q^r+1))P_{\infty})$, for all $1 \leq i \leq r)$ e $0 \leq j \leq (q^r+1)(q-1)-1$.

\medskip

With the notations on the previous section we have that:

$\bullet$ $a=q$ and $b=q^r + 1$;

$\bullet$ $P=P_{\infty}$;

$\bullet$ $div_{\infty}(x) = q P_{\infty}$ and $div_{\infty}(y) = (q^r + 1)P_{\infty}$;

$\bullet$ $H(P_{\infty}) = \langle q, q^{r}+1 \rangle$;

$\bullet$ $\rho_1 = q-1$, $\rho_2 = q^r$ and $\rho_3 = q-2$.

\medskip

Thus, using the Proposition \ref{123eta2} and the Theorem \ref{main theorem}, we can get the root diagram for one-point codes $C_{\mathcal{X}_{q^{2r}}}(D, \lambda P_{\infty})$ and then the Gr\"{o}bner basis for the module $\overline{C}$ associated to $C_{\mathcal{X}_{q^{2r}}}(D, \lambda P_{\infty})$ by Algorithm \ref{etalgoritmo}.

\subsection{A Quotient of the Hermitian curve}
Let $\mathcal{X}_{m}$ de the curve defined over $\mathbb{F}_{q^2}$ by the affine equation
   $$
y^q + y = x^m ,
  $$
  where $q$ is a prime power and $m>2$ is a divisor of $q+1$. This curve has genus $g=(q-1)(m-1)/2$, a single point at infinity, denoted by $P_{\infty}$, and others $q(1+m(q-1))$ rational points. In \cite{garcia}, it is shown that $\mathcal{X}_{m}$ ia a maximal curve and in \cite{matthews}, G. Matthews studied Weierstrass semigroup and algebraic codes over this codes. As a result present by Matthews we have that $H(P_{\infty})=\langle m, q\rangle$.

  Let $\mathbb{F}_{q^2}^{\ast} = \langle \alpha \rangle$ and $k$ such that $mk=q+1$. Then,
  
\begin{equation} \label{automorfismo tau}
\begin{array}{cccc}
\tau: & x & \rightarrow & \alpha^k x\\
& y& \rightarrow & \alpha^{q+1}y
\end{array}
\end{equation}

is an automorphism of $\mathcal{X}_{m}$ of order $m(q-1)$, which divides $q^2-1$.

It is not hard to see that under the action of the automorphism $\tau$ above the $q(1+m(q-1))$ rational points on $\mathcal{X}_{m}$ are disposed in $q+2$ orbits, where $q$ of them has length $m(q-1)$ and the remaining two orbits, one has length $q-1$ and the other has length $1$.

Taking $r=q$ and the first $r$ orbits given by points on $\mathcal{X}_m$ of the form $P=(a,b)$ with $a,b\neq0$. So, for each $i=1,\ldots,r$, given $P_{i,0}=(\alpha^{\ell_i},\alpha^{t_i}) \in O_i$, the others points $P_{i,j}$ on $O_i$ are $P_{i,j} =\sigma^j(P_{i,0})=(\alpha^{\ell_i+jk},\alpha^{t_i+j(q+1)})$, with $j \in \{ 1,\ldots, m(q-1)-1 \}$. That is,
$$
O_i= \{ P_{i,j}=(\alpha^{\ell_i+jk},\alpha^{t_i+j(q+1)}) \mbox{ ; } j=0,\ldots,m(q-1)-1 \}.
$$

Then, for $i=1,\ldots,r$ and $j = 0,1,\ldots, m(q-1)-1$, we get
$$
\displaystyle M_{i}(y) = \prod_{j=0}^{q-2}(y- \alpha^{t_i+j(q+1)})
$$
and
$$
\displaystyle B_{i,j}(x,y) = \prod_{s=0, s\neq j}^{q-2}(y- \alpha^{t_i+s(q+1)}) \prod_{s=0, s\neq j}^{m(q-1)-1} (x- \alpha^{\ell_i+sk}).
$$

So, since $div_{\infty}(x) = q P_{\infty}$ and $div_{\infty}(y) = mP_{\infty}$, follows that

\medskip

$\bullet$ $div_{\infty}(M_i(y)) = (q-1)m P_{\infty}$, that is, $M_{i}(y) \in L((q-1)m P_{\infty})$, for all $i=1,\ldots,r$;

\medskip

$\bullet$ $div_{\infty}(B_{i,j}(x,y)) = ((q-2)m+(m-1)q)P_{\infty}$, that is, $B_{i,j} \in L((m-1)q+(q-2)m)P_{\infty})$, for all $1 \leq i \leq r)$ e $0 \leq j \leq m(q-1)-1$.

\medskip

With the notations on the previous section we have that:

$\bullet$ $a=q$ and $b=m$;

$\bullet$ $P=P_{\infty}$;

$\bullet$ $(x)_{\infty} = q P_{\infty}$ and $(y)_{\infty} = mP_{\infty}$;

$\bullet$ $H(P_{\infty}) = \langle q, m \rangle$;

$\bullet$ $\rho_1 = q-1$, $\rho_2 = q-2$ and $\rho_3 = m-1$.

\medskip

Therefore, we can get the root diagram for one-point codes $C_{\mathcal{X}_{m}}(D, \lambda P_{\infty})$ and then the Gr\"{o}bner basis for the module $\overline{C}$ associated to $C_{\mathcal{X}_{m}}(D, \lambda P_{\infty})$.

\end{document}